\documentclass[11pt]{amsart}
 
\usepackage[margin=1.2in]{geometry}
\usepackage{caption}
\usepackage{subcaption}
\usepackage{hyperref}
\usepackage{graphicx}
\usepackage{caption}
\usepackage{mathtools}
\DeclarePairedDelimiter\abs{\lvert}{\rvert}%

\newtheorem{theorem}{Theorem}

\theoremstyle{definition}
\newtheorem{definition}{Definition}
\newtheorem{rem}{Remark}
\newtheorem{prop}{Proposition}
\newtheorem{lem}{Lemma}
\newtheorem{corollary}{Corollary}

\newcommand{\vv}[1]{\ensuremath{V_{#1}}}
\newcommand{\vn}[2]{\ensuremath{V_{#1}(#2)}}

\newcommand{\pp}[1]{\ensuremath{P_{#1}}}
\newcommand{\pn}[2]{\ensuremath{P_{#1}(#2)}}
\newcommand{\cc}[1]{\ensuremath{C_{#1}}}
\newcommand{\cn}[2]{\ensuremath{C_{#1}(#2)}}

\begin{document}

\title{A Regular $N$-gon Spiral}
\markright{A Regular $N$-gon Spiral}

\author{Kyle Fridberg}
\address{Department of Mathematics, Harvard University, Cambridge MA 02138}
\email{kof4@cornell.edu}

\date{14 April 2024}

\begin{abstract}
We construct a polygonal spiral by arranging a sequence of regular $n$-gons such that each $n$-gon shares a specified side and vertex with the $(n+1)$-gon in the construction. By offering flexibility for determining the size of each $n$-gon in the spiral, we show that a number of different analytical and asymptotic behaviors can be achieved. 
\end{abstract}

\maketitle

\section{Introduction}

Spirals are pervasive in fluid motion, biological structures, and engineering, affording numerous applications of mathematical spirals in modeling real-world phenomena \cite[ch.~11-14]{Thompson} \cite{Davis, Hammer}. Polygonal spirals are loosely defined as spirals that can be constructed using a sequence of polygons obeying a geometric relation. For example, in Fig. \ref{theo_spiral} we show how to join together right triangles in an infinite sequence to obtain the spiral of Theodorus---one of the oldest and most well-studied polygonal spirals \cite{Brink, Davis}. 

\begin{figure}[!htb]
    \begin{centering}
	\includegraphics[width=0.45\textwidth]{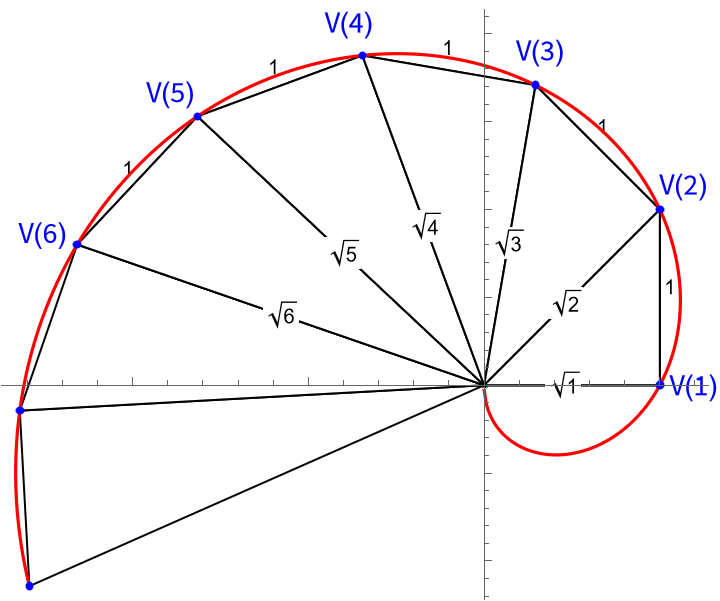}
	\caption{Plot of the first 7 triangles in the spiral of Theodorus. The sequence of shared vertex points $\vn{}{n}$ and the smooth interpolation curve proposed by Davis \cite[p.~38]{Davis} are also shown.}
    \label{theo_spiral}
    \end{centering}
\end{figure}
 In Fig. \ref{theo_spiral}, we let $\vn{}{n}$ denote the unique non-origin shared vertex of the triangle with hypotenuse $\sqrt{n}$ and the triangle with hypotenuse $\sqrt{n+1}$. Considerable attention has been given to studying the sequence $\vn{}{n}$ \cite{Brink, Davis, Gautschi, Hlawka}. The \textit{key idea} in the spiral of Theodorus construction is to define a sequence of polygons (i.e., right triangles) such that consecutive polygons share a unique side and at least one unique vertex point. By changing right triangles to similar triangles, or other types of similar polygons, we find that this idea is ubiquitous in the literature on polygonal spirals \cite{Anatriello, Strizhevsky, Yap} \cite[ch.~34]{Hammer}. 

In this paper, we make use of this key idea to construct a spiral made out of a sequence of regular $n$-gons. In particular, we start by considering the case where all the $n$-gons are normalized to have perimeter $1$ (see Fig. \ref{lam1}). We begin the construction with an equilateral triangle with side length $\frac{1}{3}$ in the first quadrant of the complex plane. Along the upper right edge of the triangle, we construct a square with side length $\frac{1}{4}$. We continue in this way, following Definition \ref{dcon}, to construct the pentagon, hexagon, etc., leading to the construction shown below. We let $V(n)$ denote the unique shared vertex point of the $n$-gon and $(n+1)$-gon in this sequence. The smooth continuation of the vertex point sequence $V(n)$ yields a new plane curve constructed from polygons (see \cite{Shikin} and \cite{Yates} for an extensive discussion of other plane curves which arise from classical geometry).

\begin{figure}[!htb]
    \begin{centering}
	\includegraphics[width= 0.6\textwidth]{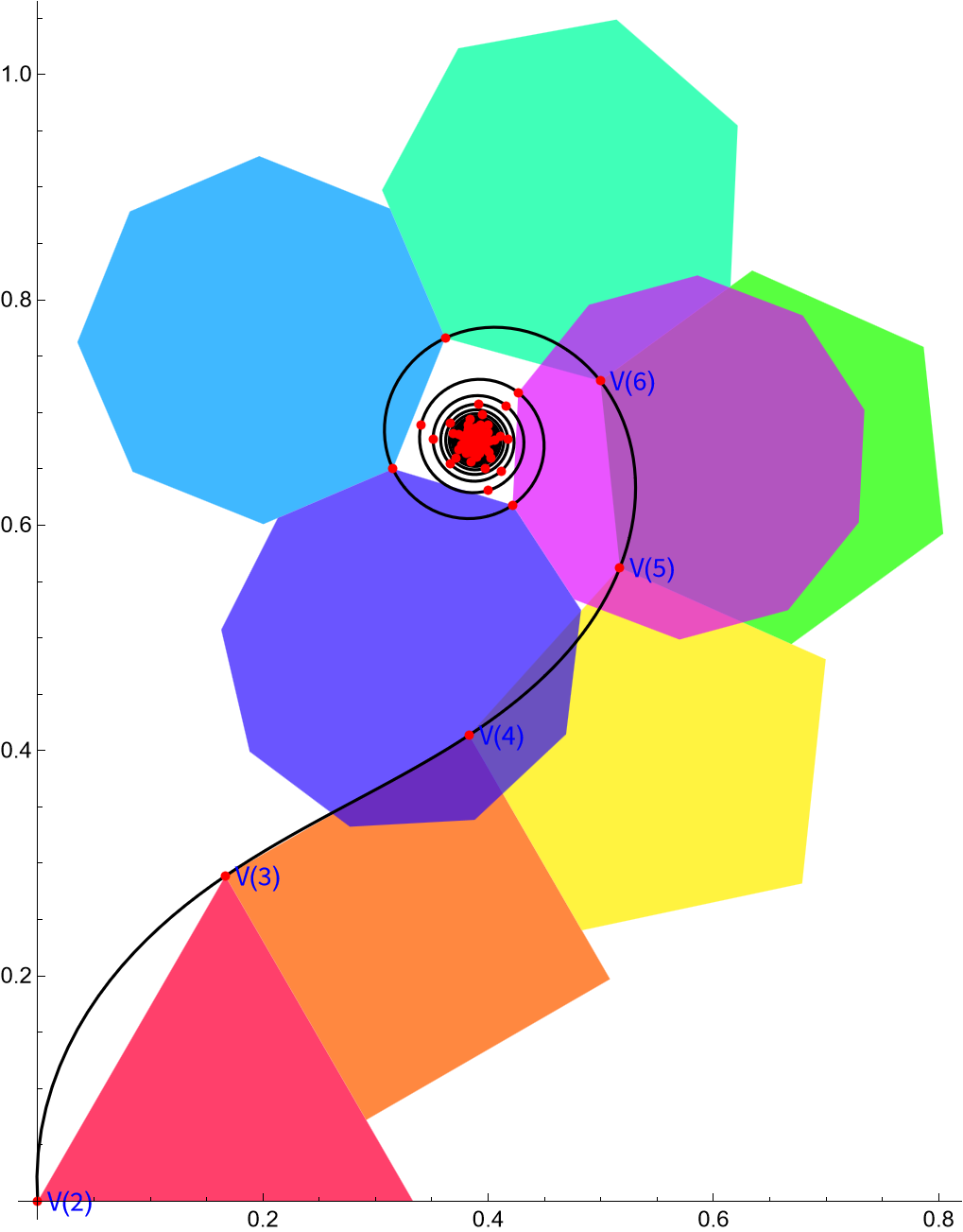}
	\caption{Plot of the first few polygons of the Definition \ref{dcon} construction (choosing $l(n) = \frac{1}{n}$). The shared vertex points $\vn{}{n}$ (see Lemma \ref{lamv}) and a smooth interpolation curve defined in Proposition \ref{interp} are also shown.}
    \label{lam1}
    \end{centering}
\end{figure}

 We desire to find the $n \to \infty$ limit behavior of the sequence $V(n)$ shown in the Fig. \ref{lam1} construction. In particular, we pose a slightly more general question: if the side length of each $n$-gon in the Fig. \ref{lam1} construction is instead given by $\ell_s(n) = n^{-s}$, for which values of $s \in \mathbb{R}$ does the sequence $\vn{\ell_s}{n}$ converge as $n \to \infty$? The answer to this question is one of the main results of this paper:

\begin{theorem}[Convergence of $\vn{\ell_s}{n}$]
\label{thm1}
\begin{flalign*}
    \text{ As } n \to \infty,  \hspace{1mm} \vn{\ell_s}{n} & \begin{cases}
    \text{converges to a point when } s > 0.\\
    \text{converges to a circular orbit when } s=0.\\
    \text{diverges when } s < 0.
    \end{cases}
\end{flalign*}
\end{theorem}

We conclude the introduction by outlining the rest of the paper. In Section \ref{construct}, we formally define the geometric construction of the $n$-gon spiral (Definition \ref{dcon}) with an arbitrary length function, leading to a formula for the shared polygon vertex points (Lemma \ref{lamv}). We also derived a formula for the sequence of polygon centers (Proposition \ref{centers}) and a curve which provides a smooth continuation of the shared vertex point sequence (Proposition \ref{interp}). In Section \ref{converge}, we use Lemma \ref{lamv} to prove Theorem \ref{thm1}, and then give a corollary to the theorem. In Section \ref{telescoping_spiral}, we introduce a special case of the $n$-gon spiral which admits particularly nice algebraic expressions, including a closed-form smooth continuation of the discrete spiral to real values of $n$ (Theorem \ref{thm_analytic}) and an appearance of the golden ratio (Remark \ref{rem1}).

\section{List of Symbols and Notations}
\label{notation}
For the convenience of the reader, we provide a list of notations and where they are defined. For example, 4, Eqn. \ref{t1} means a notation is defined in Equation \ref{t1} on page 4. Note that $\vv{l}$ and $\cc{l}$ are defined geometrically on page $3$, but their algebraic expressions appear on page $5$.

\medskip

\begin{tabular}{p{0.25\textwidth}p{0.25\textwidth}p{0.25\textwidth}p{0.25\textwidth}}
$\pp{l}$& 3, Def. \ref{dcon}&$\tilde{V}_l$& 6, Prop. \ref{interp}\\
$\vv{l}$& 3, Def. \ref{dcon}; 5, Lem. \ref{lamv} & $\ell_s$ &6, Eqn. \ref{ells}\\
$\cc{l}$&3, Def. \ref{dcon}; 5, Prop. \ref{centers}&$W$&6, Eqn. \ref{ws}\\
$\theta_n$&4, Eqn. \ref{t1}&$f$&6, Eqn. \ref{fk}\\
$H_n$&4, Eqn. \ref{Hn}&$L$&10, Thm. \ref{thm_analytic}\\
$Q_l$&5, Prop. \ref{centers}& & \\
\end{tabular}

\section{The formal construction}
\label{construct}

\begin{definition}[Construction of the discrete $n$-gon spiral $\pp{l}$]\
\label{dcon}

\begin{enumerate}
    \item Starting with $n=2$, the $n$th polygon in the spiral construction is a regular $n$-gon with side length $l(n)$, where $l: \mathbb{N} \to \mathbb{R}$ is called the length function.
    
    \item Consecutive polygons share a side and at least one vertex. Only consecutive polygons may share a side.
    
    \item Shared vertices of each polygon are consecutive vertices, and the side connecting two shared vertices is not a shared side.
    
    \item Consecutive polygons do not overlap (i.e. the shared area between consecutive polygons is 0).
    
    \item \textit{Notation:} let $\pp{l}$ represent the infinite sequence of polygons defined via $1$--$4$ above.
\begin{itemize}
    \item In this sequence, $\pn{l}{n}$ represents the $n$-gon in $\pp{l}$.
    \item Let $\vv{l}$ denote the infinite sequence of shared vertex points of $\pp{l}$. In this sequence, $\vn{l}{n}$ represents the vertex of $\pn{l}{n}$ and $\pn{l}{n+1}$ that is shared for all choices of $l(n)$.
    \item Let $\cc{l}$ denote the infinite sequence of polygon centers of $\pp{l}$. In this sequence, $\cn{l}{n}$ represents the center of $\pn{l}{n}$.
\end{itemize}
    
\end{enumerate}

\end{definition}

Once the length function $l(n)$ is chosen and the coordinates of the $2$-gon and $3$-gon are given, Definition \ref{dcon} fully determines $\pp{l}$. We will see in the derivation of Lemma \ref{lamv} that there is an algebraically simplest orientation of the construction in the complex plane, which we will adopt.

To derive a formula for the shared vertex points $\vn{l}{n}$ (Lemma \ref{lamv}), we begin by following Definition \ref{dcon} to draw an arbitrary $n$-gon and $(n+1)$-gon in the $\pp{l}$ construction. Using that the interior angles of a regular $n$-gon are $\alpha_n := \frac{\pi(n-2)}{n}$ radians, we want to find the angle with respect to the horizontal ($\theta_n$) that must be followed to get from the shared vertex point $\vn{l}{n-1}$ to $\vn{l}{n}$, pictured in Fig. \ref{angles}. In this diagram, we assume $\theta_n, \theta_{n+1} < \pi$. Equivalent results (mod $2\pi$) are obtained by considering the other relevant cases.

\begin{figure}[!htb]
    \begin{centering}
	\includegraphics[width=0.48\textwidth]{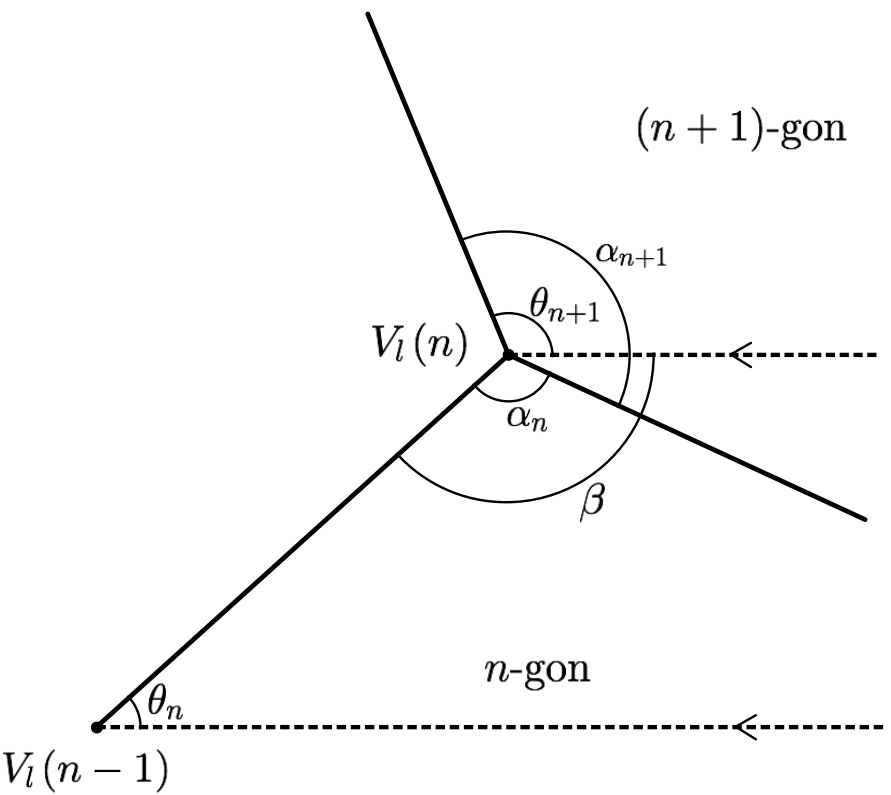}
	\caption{$\pp{l}$ angular relationships.}
    \label{angles}
    \end{centering}
\end{figure}

In the Fig. \ref{angles} diagram, we observe that the angles $\theta_n$ and $\beta = \alpha_n + \alpha_{n+1} - \theta_{n+1}$ in Fig. \ref{angles} are supplementary, which affords an angular recurrence:
\begin{equation}
     \theta_n + \beta = \pi \implies \theta_{n+1} = \theta_{n} + \frac{(n-2)\pi}{n} + \frac{(n-1)\pi}{n+1} - \pi, \hspace{2mm} n > 1. \label{t1_recurrence}
\end{equation}

To solve the recurrence for $\theta_n$, we need a starting value for $\theta_2$---a choice which determines the orientation of $\pp{l}$ in the complex plane. Appealing to simplicity, we find that choosing $\theta_2 = -3 \pi$ ensures that there is no nontrivial constant rotation applied to all elements of $\pp{l}$ (i.e. $\theta_n$ mod $2\pi$ has no constant additive term). This choice affords the orientation shown in Fig. \ref{lam1}. 

With $\theta_2$ in hand, we note that (\ref{t1_recurrence}) is a first order linear recurrence, and thus $\theta_{n+1}$ can be rewritten as a sum. Simplification of the resulting sum yields the angular relation for $\pp{l}$:
\begin{equation}
\label{t1}
    \theta_n = 2\pi\Big(\frac{n}{2}+\frac{1}{n}-2 H_n\Big), \hspace{2mm} n > 1,
\end{equation}

\noindent where $H_n$ is the $n$th harmonic number, i.e.  
\begin{equation}
    \label{Hn}
    H_n = \sum_{k=1}^n \frac{1}{k}.
\end{equation}

Using (\ref{t1}), we are now ready to write down an analytical expression for the shared vertex points $\vn{l}{n}$. Examining Fig. \ref{lam1}, we observe that each shared vertex point $\vn{l}{n}$ can be written as a the sum over the positive integers $k \leq n$ of the $k$-gon side length $l(k)$ multiplied by the corresponding rotation, $e^{i \theta_k}$. We set $\vn{l}{2} := 0$ to start the construction at the origin, as illustrated in Fig. \ref{lam1}. This is achieved by starting the sum at $k=3$. We record a formula for the shared vertex points in the following lemma.

\begin{lem}[Expression for shared vertex points of the $n$-gon spiral]
\label{lamv}
\begin{equation*}
    \vn{l}{n} = \sum_{k=3}^n l(k) \hspace{1mm} e^{i \theta_k} = \sum_{k=3}^n (-1)^k l(k) \hspace{1mm} e^{2 \pi i (\frac{1}{k} - 2H_k)}, \hspace{2mm} n \in \mathbb{N}_{> 1}.
\end{equation*}
\end{lem}

With an expression for the shared vertex points in hand, we can derive an expression for the center of each polygon in the $n$-gon spiral.

\begin{prop}[Expression for polygon centers of the $n$-gon spiral]\
\label{centers}
\begin{equation*}
    \cn{l}{n} = \vn{l}{n-1} + Q_l(n), \hspace{ 2mm} \text{where } Q_l(n) = \frac{l(n) e^{i \theta_n}}{1-e^{-\frac{2\pi i}{n}}}.
\end{equation*}
\end{prop}

\begin{proof}
Since $\pp{l}$ consists of regular polygons, there is a natural relationship between the sequences $\pp{l},\vv{l},$ and $\cc{l}$. In particular, for each $n$-gon in $\pp{l}$, there exists an isosceles triangle with vertices $\vn{l}{n-1}$, $\vn{l}{n}$, and $\cn{l}{n}$. Fix $n$. Since $\vn{l}{n-1}$ and $\vn{l}{n}$ are consecutive vertices of the $n$-gon (by Definition \ref{dcon}, part 3), the apex angle of the isosceles triangle is $\frac{2\pi}{n}$. Using basic trigonometry, we find that the two legs of the triangle have length $\frac{\left|l(n)\right|}{2\sin(\frac{\pi}{n})}$ and the base angles are $\frac{\pi}{2}-\frac{\pi}{n}$. 
Now, consider a vector $\vec{v}$ from $\vn{l}{n-1}$ to $\vn{l}{n}$. If we rotate $\vec{v}$ 
by $\frac{\pi}{2}-\frac{\pi}{n}$ and scale it to have length $\frac{\left|l(n)\right|}{2\sin(\frac{\pi}{n})}$ then the result will be a vector from  $\vn{l}{n-1}$ to $\cn{l}{n}$. The only question to consider is which direction to rotate. Since $\pp{l}$ winds counterclockwise and the polygons face outward from the spiral, rotating $\vec{v}$ clockwise will achieve the desired result.
In complex coordinates, the vectors just become numbers, and we have $v = \vn{l}{n} - \vn{l}{n-1} = l(n) e^{i \theta_n}$. Rotating $v$ by  $\frac{\pi}{2}-\frac{\pi}{n}$ radians clockwise affords $v' = l(n) e^{i \theta_n} e^{-i(\frac{\pi}{2}-\frac{\pi}{n})}$. Finally, scaling $v'$ to the proper length recovers the desired expression
\begin{equation}
     \frac{v'}{|v'|}\frac{\left|l(n)\right|}{2\sin(\frac{\pi}{n})} = \frac{l(n) e^{i \theta_n}}{1-e^{-\frac{2\pi i}{n}}} = Q_l(n).
\end{equation}
\noindent Then, by construction, we have $\cn{l}{n} = \vn{l}{n-1} + Q_l(n)$.
\end{proof}

Proposition \ref{centers} makes it clear that understanding of the behavior of $\vv{l}$ and $Q_l$ is sufficient for understanding the behavior of $\cc{l}$, and hence $\pp{l}$. \textit{As such, we will frequently refer to the sequence of shared vertex points $\vv{l}$ as ``the $n$-gon spiral."}

For the purposes of visualizing the sequences $\vv{l}$ and $\cc{l}$ and studying their behavior, it is desirable to obtain smooth continuations of the sequences to real values of $n$. Moreover, these continuations represent a new family of polygonal curves. We use a summation rearrangement trick (e.g., see section II in \cite{Albino}) to obtain a smooth continuation of $\vn{l}{n}$ to real $n$. Set $a(k):=l(k) e^{i \theta_k}$, then
\begin{equation}
\label{smoothcont}
    \vn{l}{n} = \sum_{k=3}^{\infty} a(k) - \sum_{k=n+1}^{\infty} a(k) = \sum_{k=3}^{\infty} a(k) - a(k-2+n).
\end{equation}

Assume $l: \mathbb{N} \to \mathbb{R}$ has a smooth $\mathbb{R} \to \mathbb{R}$ continuation. The $k$th harmonic number can be analytically continued to complex values of $k$ via $ H_k = \psi(k+1) + \gamma$, where $\psi(x) = \frac{\Gamma'(x)}{\Gamma(x)}$ is the digamma function and $\gamma$ is Euler's constant. As a result, we know $\theta: \mathbb{N} \to \mathbb{R}$ has the smooth $\mathbb{R} \to \mathbb{R}$ continuation $\theta_k = 2\pi\big(\frac{k}{2}+\frac{1}{k}-2(\psi(k+1)+\gamma)\big)$. Plugging in the smooth continuations of $l$ and $\theta$ into Equation \ref{smoothcont} gives the desired result:

\begin{prop}[Smooth continuation of $\vn{l}{n}$]\
\label{interp}
\begin{equation*}
    \tilde{V}_l(n) := \sum_{k=3}^{\infty} l(k) \hspace{1mm} e^{i \theta_k} - l(k-2+n) \hspace{1mm} e^{i \theta_{k-2+n}}, \hspace{2mm} n \in \mathbb{R}_{> 1}.
\end{equation*}
\end{prop}

\section{Spiral convergence}
\label{converge}

As a step toward understanding the convergence properties of the discrete $n$-gon spiral, we will analyze the limit behavior of the spiral in the case where the length function is 
\begin{equation}
    \label{ells}
    \ell_s(n) = n^{-s}, \hspace{1mm} s \in \mathbb{R}, \hspace{1mm} n \in \mathbb{R}_{>1}.
\end{equation}
 For conciseness, we let $W(s)$ denote the limit behavior of the sequence of shared vertex points,
\begin{equation}
    \label{ws}
    W(s) := \lim_{n \to \infty} \vn{\ell_s}{n} = \sum_{k=3}^{\infty} \frac{(-1)^k f(k)}{k^s},
\end{equation}
where
\begin{equation}
    \label{fk}
    f(k) := e^{2 \pi i (\frac{1}{k} - 2H_k)}.
\end{equation}

\begin{figure}[ht]
\begin{subfigure}{0.59\textwidth}
	\centering\captionsetup{width=.9\linewidth}
	\includegraphics[width=\textwidth]{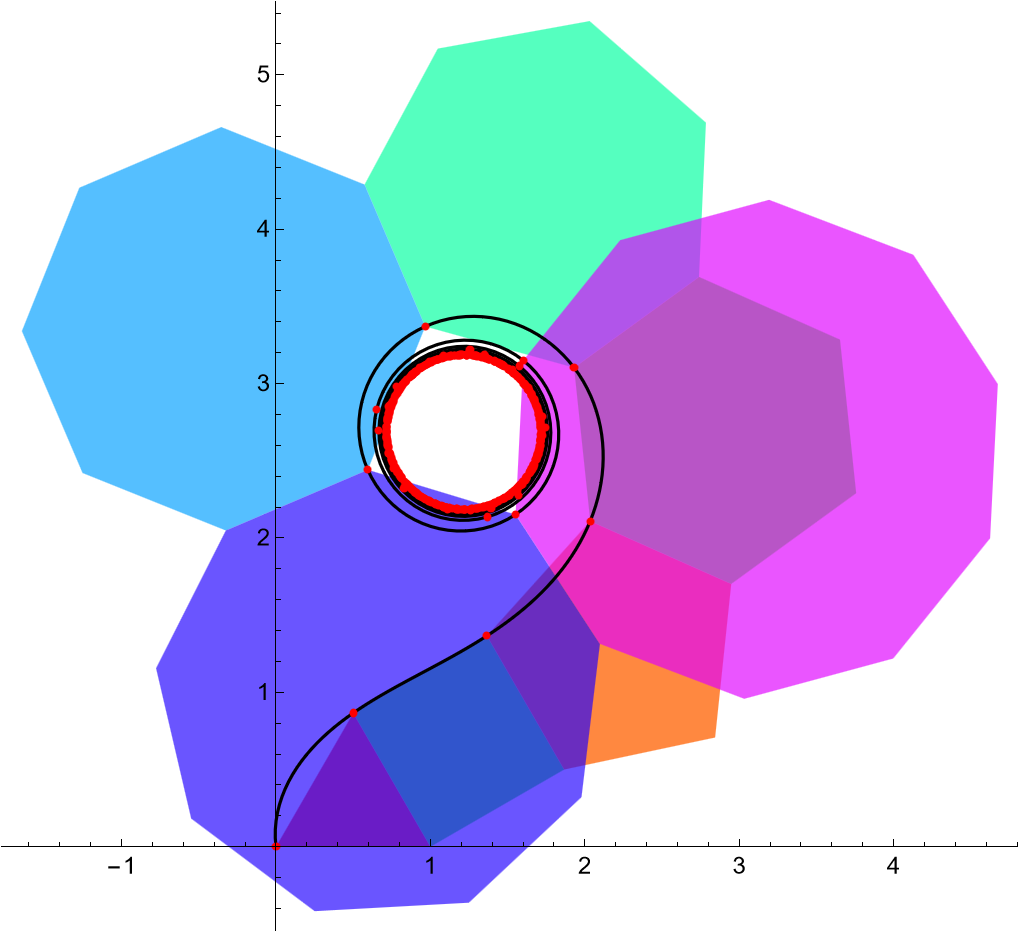}
	\caption{The first few polygons of the spiral $\pp{\ell_0}$ $(s=0)$ plotted with the sequence of shared vertices $\vv{\ell_0}$ and the interpolation curve given by (\ref{interp}). We observe the convergence of $\vv{\ell_0}$ to a circular orbit (Theorem \ref{thm1}). Using Euler's transform for alternating series, we find numerically that the center of this circular orbit is at $\lim_{s \to 0^+} W(s) = 1.21711960256553... + i \hspace{1mm} 2.68541404871695...$}
	\label{n=0_spiral}
\end{subfigure}
\begin{subfigure}{0.37\textwidth}
    \centering\captionsetup{width=.94\linewidth}
	\includegraphics[width=\textwidth]{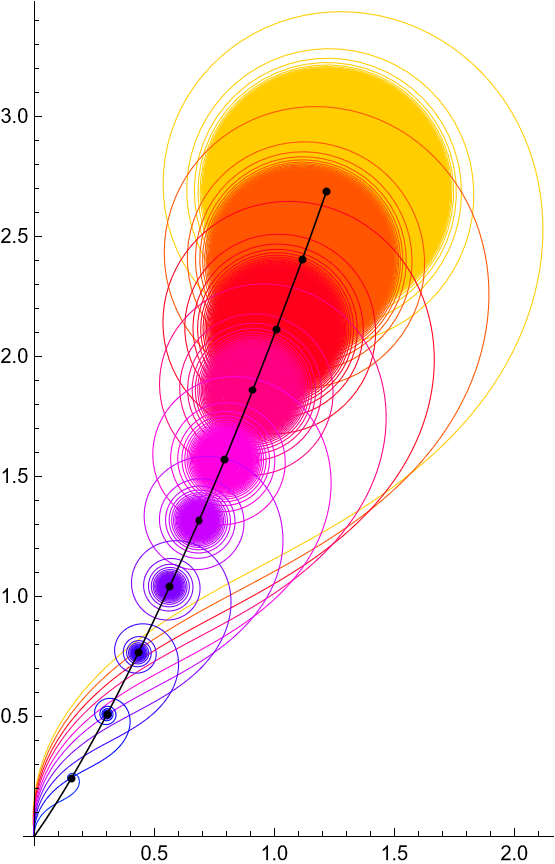}
	\caption{The  convergence points $W(s)$ for $s>0$ trace out a curve in the complex plane (black). We also plot ten spirals (from $s=0.0000726$ to $s=1.77$, interpolated by (\ref{interp})) with convergence points that are equally spaced along the curve.}
	\label{conv_curve}
\end{subfigure}
\caption{Illustration of Theorem \ref{thm1} for $s \geq 0$.}
	\label{spiral_comparison}
\end{figure}

\begin{proof}[Proof of Theorem \ref{thm1}]
We aim to characterize the basic behavior of $W(s)$ for $s \in \mathbb{R}$.

\bigskip 

\noindent \underline{Case 1: $s < 0$.} The individual terms of the series are unbounded, hence $W(s)$ diverges.

\bigskip

\noindent \underline{Case 2: $s > 1$.} $W(s)$ is an absolutely convergent series. 

\bigskip

\noindent \underline{Case 3: $0 < s \leq 1$.} We first accelerate convergence of $\vn{\ell_s}{n}$ by adding consecutive pairs of terms:
\begin{equation}
\label{Saccel}
    \vn{\ell_s}{2n} = \sum_{k=3}^{2n} \frac{(-1)^k f(k)}{k^s} = \sum_{j=2}^{n} F(j), \text{ where } F(j) := \frac{(2j-1)^sf(2j)-(2j)^sf(2j-1)}{((2j-1)(2j))^s}.
\end{equation}

\noindent We now manipulate $F(j)$ in order to show $\vn{\ell_s}{2n}$ is absolutely convergent. To this end, we invoke the identity $H_{2j-1} = H_{2j} - \frac{1}{2j}$ to write $f(2j-1)$ in terms of $f(2j)$:
\begin{equation}
    f(2j-1) = \exp\Big(2\pi i\Big(\frac{1}{2j-1}+\frac{1}{2j}\Big)\Big) f(2j). \label{f2j-1}
\end{equation}

\noindent Plugging (\ref{f2j-1}) and $(2j-1)^s = (2j)^s(1-\frac{1}{2j})^s$ into $F(j)$, we multiply the numerator and denominator of $F(j)$ by $2j-1$, which affords
\begin{equation}
\label{Fj}
    F(j) = \frac{f(2j)\Big((2j-1)\big(1-\frac{1}{2j}\big)^s - (2j-1) \exp\big(2\pi i\big(\frac{1}{2j-1}+\frac{1}{2j}\big)\big)\Big)}{(2j-1)^{1+s}} =: \frac{T(j)}{(2j-1)^{1+s}}.
\end{equation}

\noindent Since $\sum_{j=2}^{\infty} \frac{1}{(2j-1)^{1+s}} = \zeta(1+s)(1-2^{-1-s})-1$ is convergent for $s \in (0,1]$, showing $\abs*{T(j)}$ is bounded gives the desired convergence result. By the triangle inequality,
\begin{flalign}
    \abs*{T(j)} & < \max \hspace{1mm}\abs*{ (2j-1)\Big(1-\Big(1-\frac{1}{2j}\Big)^s\Big)} + \max \hspace{1mm}\abs*{-(2j-1)\Big(1-\exp\Big(2\pi i \Big(\frac{1}{2j-1} + \frac{1}{2j}\Big)\Big)\Big)} \nonumber\\
    & =: \max(A(j,s)) + \max(B(j)).
    \label{Tj}
\end{flalign}

\begin{lem}
\label{maxA}
\begin{equation*}
    \max(A(j,s)) = \lim_{j\to\infty} A(j,s) = s.
\end{equation*}
\end{lem}

\begin{proof}
We obtain the $j \to \infty$ limit by direct calculation. To show this is a maximum, we make use of the following Bernoulli-type inequality (which can be readily proved using the generalized binomial theorem):
\begin{equation}
\label{bern}
    1-sy < (1+y)^{-s} < 1 \text{ for } y \in (0,1), \hspace{1mm} s \in (0,1].
\end{equation}
Plugging $y = \frac{1}{2j-1}$ into (\ref{bern}), we simplify to obtain:
\begin{equation}
\label{bernsimp}
    \frac{s}{2j-1} > 1-\Big(1+\frac{1}{2j-1}\Big)^{-s} > 0.
\end{equation}

\noindent Observing that $(1+\frac{1}{2j-1})^{-s} = (1-\frac{1}{2j})^s$, we multiply (\ref{bernsimp}) by $2j-1$ to afford the desired bound:
\begin{equation}
    s > A(j,s) = (2j-1)\Big(1-\Big(1-\frac{1}{2j}\Big)^s\Big) > 0.
\end{equation}
\end{proof}

\begin{lem}
\label{lemBj}
\begin{equation*}
    \max(B(j)) = \lim_{j \to \infty} B(j) = 4\pi.
\end{equation*}
\end{lem}

\begin{proof}
The $j \to \infty$ limit is obtained by direct calculation. Setting $x(j):=\pi \big(\frac{1}{2j-1}+\frac{1}{2j}\big)$, we have that $x(j) \in (0, \pi)$ for all $j \geq 2$ since $x(2) < \pi$ and $x(j)$ is a nonnegative monotonically decreasing function. From this, we can simplify $B(j)$ to obtain
\begin{equation}
    B(j):= \abs*{(2j-1)(1-\exp(2 i x))} = 2(2j-1) \sin(x) \text{ for } j \in [2,\infty).
\end{equation}
Calculation of $B'(j)$ reveals
\begin{equation}
    B'(j) = \frac{\pi}{j^2} + 4(\sin x - x \cos x).
\end{equation}
Since $\sin x - x \cos x$ is positive on $(0, \pi)$, $B(j)$ is increasing for all $j \geq 2$.
\end{proof}
 
 \noindent Plugging Lemmas \ref{maxA} and \ref{lemBj} into (\ref{Tj}) and (\ref{Fj}) affords
 \begin{equation}
     \sum_{j=2}^{\infty} \abs*{F(j)} < (4 \pi + s)\big(\zeta(1+s)(1-2^{-1-s})-1\big),
 \end{equation}
 where $\zeta(1+s) < \infty$ for $s > 0$, hence $\sum_{j=2}^{\infty} F(j)$ is absolutely convergent for $s \in (0,1]$. Since the summand of $\vn{\ell_s}{n}$ vanishes as $n \to \infty$, $\lim_{n \to \infty} \vn{\ell_s}{2n+1} = \lim_{n \to \infty} \vn{\ell_s}{2n} = \sum_{j=2}^{\infty} F(j) = W(s)$, so $W(s)$ is convergent.

\bigskip

\noindent \underline{Case 4: $s = 0$}. As in Case 3, we accelerate convergence of $\vn{\ell_0}{n}$ by pairing terms:
\begin{equation}
    \vn{\ell_0}{2n} = \sum_{j=2}^{n} \big( f(2j) - f(2j-1) \big) =: U(n). \label{Saccel2}
\end{equation}

Let $r \geq 1$ be a rational constant such that $nr \in \mathbb{N}$. We desire to show that the distance between two arbitrary, evenly indexed spiral points $U(nr)$ and $U(n)$ is a sinusoidal function of $r$ in the $n \to \infty$ limit. Hence, we desire to calculate $B(r) := \lim_{n \to \infty} \abs*{U(nr)-U(n)}$.

 Plugging the asymptotic harmonic series expansion $H_n = \gamma + \log(n) + \frac{1}{2n} + \mathcal{O}(n^{-2})$ into $B(r)$, we simplify and apply the Euler-Maclaurin formula to obtain
\begin{flalign}
    B(r) & = \lim_{n \to \infty} \abs*{\sum_{j=n+1}^{nr} \Big(e^{-4 \pi i \log(2j) + \mathcal{O}(\frac{1}{j^2})} - e^{-4 \pi i \log(2j-1) + \mathcal{O}(\frac{1}{j^2})}\Big)}\\
    & = \lim_{n \to \infty} \Bigg\lvert \int_{n}^{nr} \Big(e^{-4\pi i \log(2x)} - e^{-4\pi i \log(2x-1)}\Big)dx + \text{error terms} \hspace{1mm}\Bigg\rvert. \label{EMbound}
\end{flalign}

Using standard bounds \cite{Lehmer}, we show that the error terms vanish as $n \to \infty$. Integrating \ref{EMbound} affords
\begin{equation}
    B(r) = \lim_{n \to \infty} \Bigg\lvert \Bigg[\frac{e^{(1-4\pi i) \log(2x)}-e^{(1-4\pi i) \log(2x-1)}}{2(1-4\pi i)}\Bigg]_{n}^{nr} \Bigg\rvert.\label{Blim}
\end{equation}

Expanding out the expression for $B(r)$, we make use of Euler's formula to separate real and imaginary parts of the complex exponential terms. It is helpful to assign variables to the following expressions for clarity:
\begin{flalign}
    X(x) & := e^x\cos(4\pi x), \hspace{2mm} Y(x) := e^x\sin(4\pi x) \label{vars1} \\
    a & := \log \big(2nr\big), \hspace{2mm} b := \log \big(2nr - 1\big) \label{vars2} \\
    c & := \log(2n), \hspace{2mm} d := \log(2n-1). \label{vars3}
\end{flalign}

After collecting all the real and imaginary parts, we can write down Equation \ref{limexpr}. Using the Maclaurin series $\log(1-\epsilon) = - \epsilon + \mathcal{O}(\epsilon^2)$, we obtain the large $n$ asymptotics $b = \log(2nr) + \log(1-\frac{1}{2nr}) \approx \log(2nr) - \frac{1}{2nr}$ and $c = \log(2n) + \log(1-\frac{1}{2n}) \approx \log(2n) - \frac{1}{2n}$. Plugging these expressions into \ref{limexpr}, we obtain $X(a) - X(b) \approx 2nr \big[\cos(4\pi \log(2nr))-\big(1-\frac{1}{2nr}\big)\cos(4\pi \log(2nr) -\frac{4\pi}{2nr})\big]$, for example. We then apply the angle sum formulas for $\cos(\alpha - \beta)$ and $\sin(\alpha - \beta)$ to expand the relevant $\sin$ and $\cos$ terms. Following this, we use the truncated Maclaurin series $\cos(\epsilon) \approx 1 - \frac{\epsilon^2}{2}$ and $\sin(\epsilon) \approx \epsilon$ to eliminate all $\sin$ and $\cos$ terms which do not have argument $4\pi \log(2nr)$ or $4\pi \log(2n)$. After squaring the two terms under the square root and simplifying using standard trigonometric identities, we obtain $B(r) = \lim_{n \to \infty} \sqrt{\sin^2(2\pi \log(r)) + \mathcal{O}(\frac{1}{n})}$, which yields Equation \ref{Blim1}. The Equation \ref{limexpr} limit can be evaluated in Mathematica for verification. 

\begin{flalign}
    B(r) & = \abs*{\frac{1}{2(1-4\pi i)}} \lim_{n \to \infty} \sqrt{\big(X(a)-X(b)-X(c)+X(d)\big)^2 + \big(Y(a)-Y(b)-Y(c)+Y(d)\big)^2} \label{limexpr}\\
    & = \abs*{\sin(2\pi \log(r))} .\label{Blim1}
\end{flalign}

As claimed, the distance between arbitrary points $U(nr)$ and $U(n)$ in the $n \to \infty$ limit is sinusoidally dependent on the distance parameter $r$. Since $U(n) = \vn{\ell_0}{2n}$, Equation \ref{Blim1} shows that all evenly indexed points on the spiral converge to a circle with diameter 1 (Fig. \ref{n=0_spiral}). By explicitly constructing the circle through $\vn{\ell_0}{2n-2}, \vn{\ell_0}{2n},$ and $\vn{\ell_0}{2n+2},$ it is straightforward to show that $\vn{\ell_0}{2n+1}$ is on this circle in the $n \to \infty$ limit, which confirms that the odd spiral points converge to the same circle.

\bigskip

\noindent This concludes the proof of Theorem \ref{thm1}.

\end{proof}

Theorem \ref{thm1} provides the limiting behavior of $\vn{\ell_s}{n}$ for $\ell_s(n) = n^{-s}$, which may be used to find the limiting behavior of $V(n)$ for geometrically significant length functions which are asymptotic to $\ell_s$. For example, we can consider a Definition \ref{dcon} spiral construction where the length function is determined by each $n$-gon being inscribed or circumscribed inside a circle of radius $n^{-s}$, $s \in \mathbb{R}$.

\begin{corollary}[Corollary of Theorem \ref{thm1}]
Let $\pp{insc}$ and $\pp{circ}$ denote the $n$-gon spiral constructions where each $n$-gon is inscribed and circumscribed with respect to a circle of radius $n^{-s}$, $s \in \mathbb{R}$, respectively. Then $\lim_{n \to \infty} \vn{insc}{n}$ and $\lim_{n \to \infty} \vn{circ}{n}$ converge for $s > -1$. 
\end{corollary}

\begin{proof}
The side lengths of the inscribed and circumscribed $n$-gons are given by $\ell_{insc}(n) = 2n^{-s} \sin(\frac{\pi}{n}) < \frac{2\pi}{n^{1+s}}$ for all $n > 0$ and $\ell_{circ}(n) = 2n^{-s} \tan(\frac{\pi}{n}) < \frac{2\pi}{n^{1+s}} + \frac{8 \pi^3 \sqrt{3}}{n^{2+s}}$ for all $n \geq 3$, respectively. The $\ell_{circ}(n)$ bound follows from computing an upper bound on the remainder term of the corresponding Taylor series. The result follows by Theorem \ref{thm1}.
\end{proof}

\noindent A similar result can be shown for the $n$-gon spiral where each $n$-gon has area $n^{-s}$.

\section{The telescoping spiral}
\label{telescoping_spiral}

Here, we present a special choice of the length function that leads to closed form formulae for the discrete spiral points as well as their smooth continuation. 

\begin{theorem}[The telescoping spiral]
\label{thm_analytic}\

\noindent The $n$-gon spiral with length function $L(k) = 2 \cos (\frac{2\pi}{k})$ admits the following closed form smooth continuation for $n \in \mathbb{R}_{>1}$:
\begin{flalign}
    \vn{L}{n} & = -1 + (-1)^n e^{-4 \pi i(\gamma + \psi(n+1))}, \nonumber\\
    Q_L(n) & = (-1)^n e^{-4 \pi i(\gamma + \psi(n))} \bigg(\frac{1+e^{-\frac{4 \pi i}{n}}}{1-e^{-\frac{2 \pi i}{n}}}\bigg). \nonumber
\end{flalign}

\end{theorem}

\begin{figure}[!h]

\begin{subfigure}{.585\textwidth}
	\centering\captionsetup{width=\linewidth}
	\includegraphics[width=\textwidth]{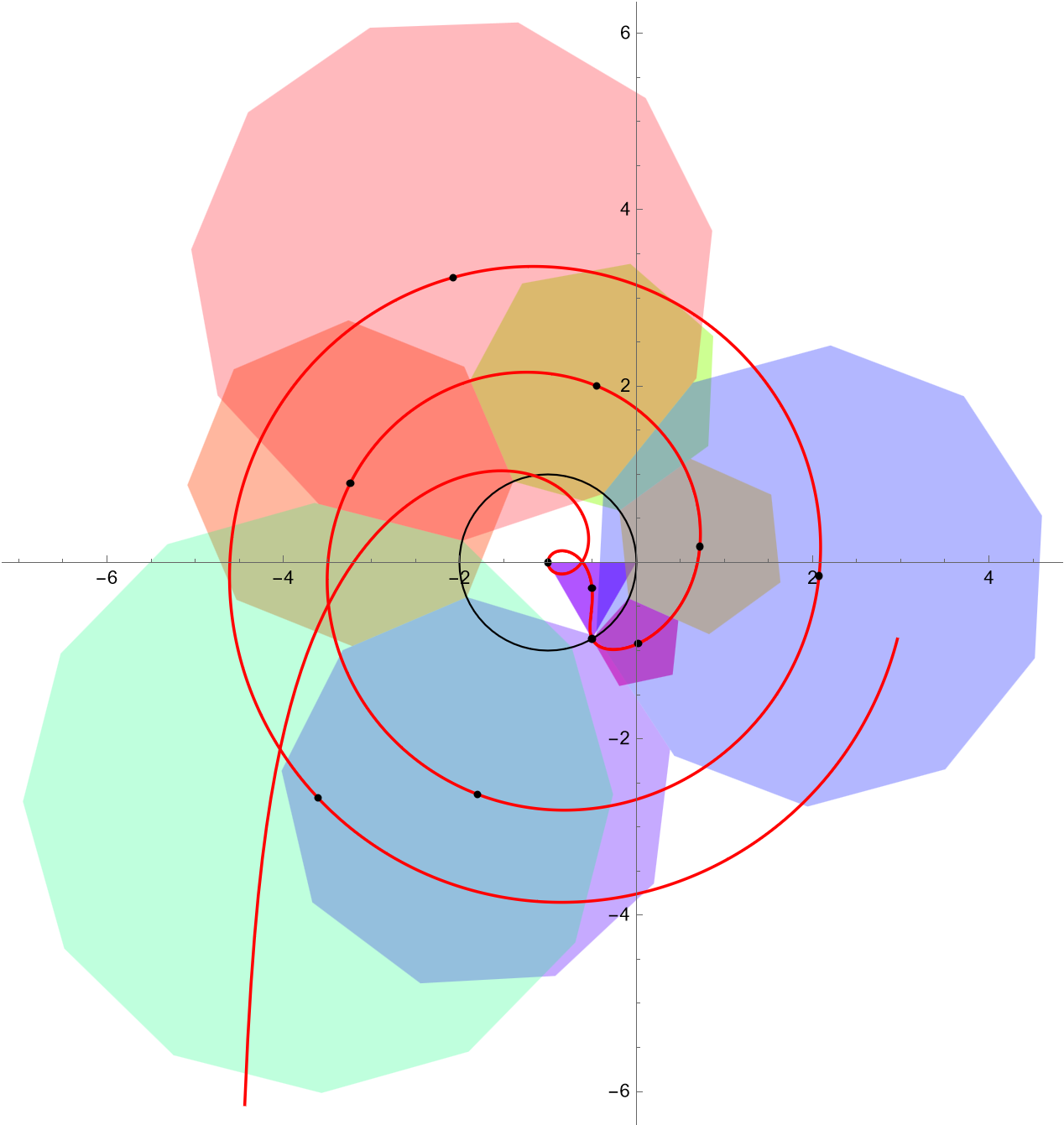}
\caption{Plot of the polygon spiral $\pp{L}$ up to the $12$-gon. $\vn{L}{n}$ traces out the unit circle centered at $-1$. The polygon centers are also marked, with their smooth continuation $\cn{L}{n} = \vn{L}{n-1} + Q_L(n)$ plotted from $n = 1.05$.}
	\label{degen}
\end{subfigure}
\begin{subfigure}{.40\textwidth}
    \centering\captionsetup{width=0.8\linewidth}
	\includegraphics[width=0.9\textwidth]{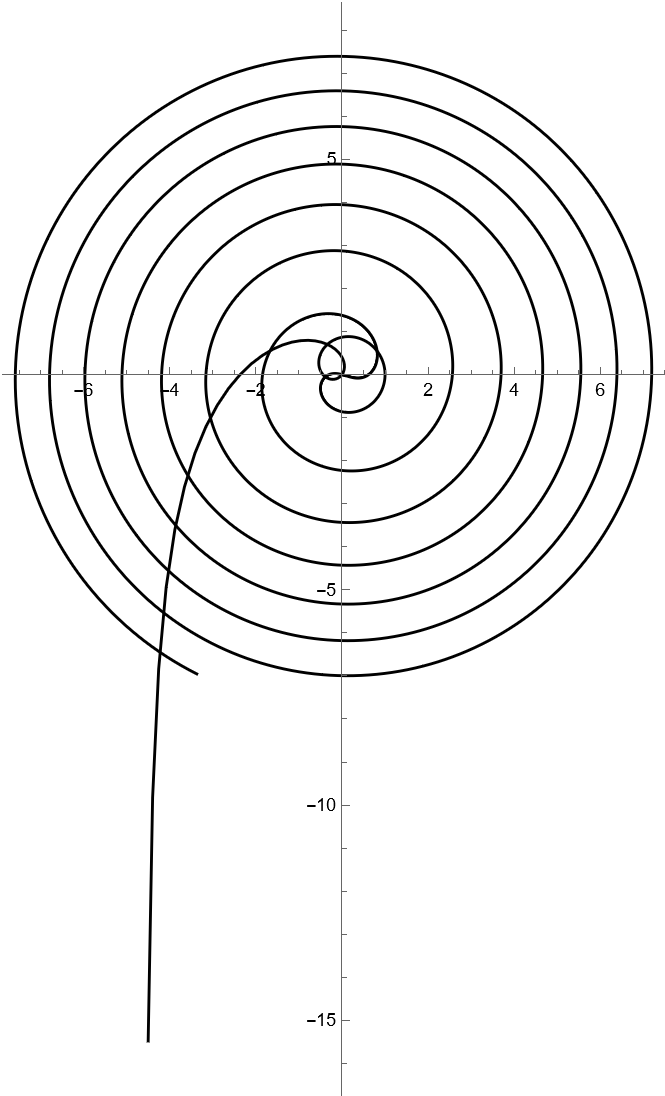}
	\caption{$Q_L(n)$ plotted from $n=1.02$ to $n=25$. In the limit as $n \to 1$, Re$(Q_L(n)) \to 2(1-\frac{\pi^2}{3})$.\vspace{4mm}}
	\label{Qspiral}
\end{subfigure}

	\caption{The telescoping spiral construction and smooth continuation.}
	\label{tele_spiral}
	
\end{figure}

\begin{proof}
Writing $L(k)$ in exponential form, we plug it into Lemma \ref{lamv} and simplify to obtain a telescoping series!
\begin{equation}
\label{tele1}
    \vn{L}{n} = \sum_{k=3}^n (-1)^k \big(e^{-4 \pi i H_{k-1}} + e^{-4 \pi i H_{k}} \big) = -1 + (-1)^n e^{-4 \pi i H_n}.
\end{equation}

The continuation of the Harmonic numbers to real values of $n$ via $ H_n = \gamma + \psi(n+1)$ provides a direct continuation of $\vn{L}{n}$ to $n \in \mathbb{R}_{>1}$, where $n \leq 1$ is excluded because we cannot construct a $1$-gon to obey Definition \ref{dcon} (and $\cn{L}{1}$ is not defined). By plugging $L(n)$ into Proposition \ref{centers}, we also obtain a closed-form smooth continuation of $Q_L(n)$, and hence $\cn{L}{n}$. 

\end{proof}

From Theorem \ref{thm_analytic}, we see that all the values of $\vn{L}{n}$ lie on the unit circle centered at $-1$ (Fig. \ref{degen}), hence $\vn{L}{n}$ forms a degenerate spiral for $n \in (1, \infty)$. On the interval $n \in (1,\infty)$, $L(n)$ is negative for $n \in (\frac{4}{3}, 4)$, zero at $n= \frac{4}{3}, 4$, and positive everywhere else. At the zeros of $L(n)$,  the $n$-gon centers and vertices coincide (Fig. \ref{degen}) and $Q_L(n) = 0$ (Fig. \ref{Qspiral}).

\begin{rem}[Golden ratio intersection]
\label{rem1}
In Fig. \ref{degen}, we observe that $\cn{L}{n}$ intersects itself inside the unit disk centered at $-1$. This intersection occurs at $n = \varphi := \frac{\sqrt{5} + 1}{2}$ and $n = \varphi + 1$. At this point, $\cn{L}{\varphi} = -i e^{-\pi i(4(\gamma + \psi(\varphi))+\varphi)} \cot(\pi \varphi) - 1$.
\end{rem}

\section{Further Directions}

There is much additional work to be done toward understanding the $n$-gon spiral introduced here. We showed that the sequence of shared polygon vertex points in Fig. \ref{lam1} converges, but is it possible to find a closed-form expression for this convergence value? Additionally, Theorem \ref{thm1} only describes the $n \to \infty$ behavior of the $n$-gon spiral for a particular family of length functions. Can one develop a convergence condition that applies to arbitrary length functions?

There is a long history of studying curves that originate from geometric constructions (see \cite{Yates} for historical notes). In Proposition \ref{interp}, we give a formula for a smooth curve interpolating the points of the sequence $\vn{l}{n}$. Studying how this curve compares to classical spirals in the literature is a natural follow up question. Additionally, the closed-form smooth continuation of the telescoping spiral (Theorem \ref{thm_analytic}) makes this spiral a particularly attractive choice for further study. We identified self-intersections of $Q_L(n)$ and $\cn{L}{n}$ at $n = \frac{4}{3}, 4$ and $n = \varphi, \varphi + 1$, respectively---do these spirals have other self-intersections at algebraic values of $n$? Can we use the smooth continuation of $\vv{L}$ and $\cc{L}$ to define a natural, continuous geometric transformation from a regular $n$-gon to a regular $(n+1)$-gon? If so, what would a ``regular polygon'' with a non-integer number of sides look like?

The notion of arranging polygons with increasing numbers of sides is not limited to the polygonal spiral discussed here. In particular, making changes to each of the rules of Definition \ref{dcon} offers flexibility for discovering a range of intriguing polygonal constructions. 

\section{Acknowledgements}
I would like to thank Ryan Mike for helpful discussions concerning this article. I would like to thank the anonymous referee for providing comments which significantly improved the quality of the presentation. I am especially grateful to Dusty Grundmeier for his constructive feedback and mentorship during the process of writing this article.

\end{document}